\documentclass[12pt]{article}   	
\usepackage{geometry}                		
\usepackage{graphicx}				
\usepackage{caption}
\usepackage{subcaption}
\usepackage{amssymb}
\usepackage{amsthm}
\usepackage{amscd}
\usepackage{csquotes}
\usepackage{amsmath}
\usepackage{mathtools}
\usepackage{float}
\usepackage{physics}
\MakeOuterQuote{"}
\newtheorem*{theorem*}{Theorem}

\newtheorem{theorem}{Theorem}[section]
\newtheorem{proposition}[theorem]{Proposition}
\newtheorem{lemma}[theorem]{Lemma}

\newtheorem*{conjecture*}{Conjecture}
\title{Ill-Posedness of the Euler Equations Linearized around Homogeneous Steady States}
\author{Matei P. Coiculescu\thanks{
Department of Mathematics, New York University,
Email: \texttt{mpc9525@nyu.edu}
}}
\begin{document}
\maketitle
\begin{abstract}
Let $L^2_m(\mathbb{R}^2)$ be the space of square-integrable functions on $\mathbb{R}^2$ with $m$-fold rotational symmetry. Let $\overline{\omega}(r,\theta) = r^{-\alpha}f(\theta)$ be a homogeneous steady state of the two-dimensional incompressible Euler equations with $(m\cdot l)$-fold rotational symmetry. If $f(\theta)$ is a constant function we say that $\overline{\omega}$ is a radial power-law vortex. We prove that the incompressible Euler equations in vorticity form, linearized around any homogeneous steady state $\overline{\omega}$ that is not a radial power-law vortex, are ill-posed on $L^2_m(\mathbb{R}^2)$ for any $m\geq 2$, any $l\geq 1$ and $\alpha\in (0,1)$. \end{abstract}
\section{Introduction}
We consider the incompressible Euler equations in the plane:
\begin{equation}
\label{EULER}
\begin{gathered}
\partial_t \omega + v\cdot\nabla \omega=0\\
v = K_{BS}\ast \omega\\
\omega(t=0) = \omega_0.
\end{gathered}
\end{equation}
Here $K_{BS}$ denotes the Biot-Savart kernel in two dimensions. We work with the vorticity form of the equations. One of the many motivations for studying Equation \eqref{EULER} is to understand its initial value problem from unbounded initial vorticity. A standard class of initial data is vorticity in $L^1\cap L^p$ where $1<p\leq \infty$. While existence of global-in-time solutions is known for all $p>1$, uniqueness is only known for bounded vorticity in $L^1\cap L^\infty$ or for vorticity distributions with some Osgood-type condition.

Vishik, in a pair of works \cite{V1}, \cite{V2}, provided the first concrete evidence for non-uniqueness of solutions from unbounded vorticity. In particular he proved that if one admits a force in $L^1_t (L^1\cap L^p)_x$ for $2<p<\infty$, then non-uniqueness may also occur (even from zero initial data). The main idea behind Vishik's argument is to consider the initial value problem in self-similar coordinates, where linear instability of a self-similar profile corresponds directly to non-uniqueness of a solution from singular data. See \cite{DM} for a concise description of this non-uniqueness mechanism for the Euler equations. Due to the difficulty of studying the linearization of the Euler equations in self-similar variables, Vishik instead proves that an instability exists in physical coordinates and proceeds with a singular limit argument to carry the instability through to the self-similar equation. A comprehensive review of Vishik's entire argument is given in the monograph \cite{ABC}.

While ingenious, several issues appear when one attempts the same argument without a force. For instance, one should first find exact forward self-similar solutions to Equation \eqref{EULER}, and, indeed, this is already a difficult question. Any such solution arises from a homogeneous initial data, and, until recently, the only known self-similar solutions to Equation \eqref{EULER} were precisely the stationary homogeneous solutions and small perturbations of the power-law vortex with vorticity $|x|^{-\alpha}$. We refer the reader to \cite{C}, \cite{CC}, and \cite{LS} for references on homogeneous solutions to the Euler equations, and we refer the reader to \cite{H} for the first construction of self-similar solutions from arbitrary $(-\alpha)-$homogeneous vorticity data, when $\alpha\in (1,2)$. We note, however, that the construction of self-similar solutions from less singular arbitrary data is still open. The purpose of this note is to essentially rule out the use of homogeneous steady states for the non-uniqueness program, which further illustrates the difficulties lying in Vishik's approach.

Let $A:D(A)\subset X\to X$ be a closed linear operator. We recall from \cite{EN} that an abstract Cauchy problem of the form
$$
\begin{gathered}
\dot{u}(t) =Au(t) \quad \forall t\geq 0\\
u(0)=x
\end{gathered}
$$
is called \textit{well-posed} if and only if the domain $D(A)$ is dense and for each $x\in D(A)$ there exists a unique solution $u(t)$ that continuously depends on the initial data in the sense that for any sequence $x_n\in D(A)$ with $x_n\to 0$ we have $u(t,x_n)\to 0$ uniformly in compact intervals of time. It is proven in Theorem 6.7 of Chapter II in \cite{EN} that the abstract Cauchy problem above is well-posed if and only if $A$ is a generator of a strongly continuous semigroup. We call the Cauchy problem \textit{ill-posed} if and only if it is not well-posed. We shall always work with Hilbert spaces of functions on $\mathbb{R}^2$ for which there is a natural orthogonal decomposition into angular modes. For example we have
$$L^2(\mathbb{R}^2) = \bigoplus^{\ell^2}_{k\in\mathbb{Z}} U_k,$$
where
$$U_k := \{f(r) e^{ik\theta} :  f\in L^2(\mathbb{R}^+, rdr)\}.$$
We also consider the spaces $L^2_m(\mathbb{R}^2)$ of square-integrable functions on $\mathbb{R}^2$ that are $m$-fold rotationally symmetric. In terms of the notation above, we have the orthogonal decomposition:
$$L^2_m(\mathbb{R}^2) = \bigoplus^{\ell^2}_{k\in\mathbb{Z}} U_{km}.$$
One important reason to work with these subspaces of symmetric functions is that the Biot-Savart law is better behaved than on $L^2(\mathbb{R}^2)$, where the Biot-Savart operator is simply unbounded (see \cite{ABC} and \cite{BC} for details). We also note here that for a homogeneous steady state to have locally finite kinetic energy one requires the homogeneity parameter $\alpha<2$ and for a homogeneous steady state to have decaying vorticity, one requires $\alpha>0$. On the other hand, for the background vorticity profile to be in $L^1_{loc}(\mathbb{R}^2)$, we require $\alpha<1$, as we shall observe in the next section. Henceforth, we choose the parameter $\alpha\in (0,1)$. The main objects of study here are the homogeneous steady states of the two-dimensional Euler equations whose stream functions vanish along some ray from the origin. If the stream-function does not vanish along any such ray, the homogeneous solution is called {\textit{elliptic}}. Due to the classification theorem of Luo and Shvdkoy from \cite{LS} and their addendum in \cite{LS2}, we know that for all $\alpha\in (0,3/2)\cup (3/2,2)$, any homogeneous steady state that is not the power law vortex is not elliptic, i.e. the stream function of the solution has the vanishing property. Moreover, as we shall see, it is possible to obtain asymptotic formulae for the singular behavior of the steady state along these rays. When $\alpha=3/2$, there is a family of exact steady states:
$$\overline{\psi}(r,\theta) = r^{1/2}\sqrt{\gamma_1+\gamma_2\cos(\theta)},$$
where $|\gamma_2|\leq \gamma_1$, which we shall not analyze. When $\gamma_1>|\gamma_2|$, the solution is elliptic. There does not seem to be a canonical term for these solutions, but we note that Kaden in \cite{K} seems to be the first to study vortex-sheet solutions of the Euler equations with data having this homogeneity (on the other hand, Kaden does not appear to study this family of steady states). When $\alpha\in (3/2,2)$ the only solutions are the power-law vortex and the power-law shear flow with stream function $\overline{\psi}(r,\theta) = r^{2-\alpha}|\cos(\theta)|^{2-\alpha}$. Also, we exclude the endpoint case $\alpha=1$ from our analysis since this corresponds to a vortex sheet solution for which the distributional calculus becomes inconvenient.

 The main result of this note is now:
\begin{theorem}
\label{MAIN}
The two-dimensional incompressible Euler equations in vorticity form, linearized around any $(m\cdot l)$-fold symmetric homogeneous steady state that is not a radial power-law vortex, are ill-posed on $L^2_m(\mathbb{R}^2)$ for any $m\geq 2$, $l\geq 1$, and homogeneity parameter $\alpha\in (0,1)$.
\end{theorem}
We only require the symmetry condition on the background to ensure that, at the very least, the linearization maps $m$-fold symmetric functions into a space of $m$-fold symmetric functions. We recall that there exists homogeneous steady states with $m$-fold rotational symmetry for any $m\geq 2$, see \cite{CC} for details.
The origin of the ill-posedness for the linearization around non-radial-vortex solutions is the fact that the linearization is not densely defined as an operator on $L^2_m(\mathbb{R}^2)$: the maximal domain is not dense. In particular, the linearization around such background profiles cannot generate strongly continuous semigroups. This sort of ill-posedness should therefore also occur for the linearization in self-similar coordinates, but we do not pursue this question here. Since the ill-behavior is due to the presence of a vorticity singularity of the background on a ray from the origin, we also conjecture that all the symmetry assumptions above can be dropped and that there is ill-posedness in $L^2(\mathbb{R}^2)$ for any homogeneous steady state with the vanishing property. Note again that the homogeneous steady-states are stationary solutions without force in both physical and self-similar versions of the Euler equations.

Since one certainly needs at least the well-posedness of the linear problem to proceed along the lines of Vishik's proof, Theorem \ref{MAIN} shows that the only remaining case is the study of the $(-\alpha)$-homogeneous power-law vortex and the solutions when $\alpha\geq 1$. On the other hand, the present author and Tim Binz in \cite{BC} proved the linear stability of the power-law vortex with some mild assumptions on the homogeneity parameter and symmetry class in both physical and self-similar coordinates, and it is unclear how to overcome the fact that the background vorticity is not in $L^1_{loc}$ when $\alpha>1$, if one is to analyze the linearization in vorticity form. Thus, any attempt to use Vishik's approach in the unforced case should probably rely on more arbitrary data, like the data considered in \cite{H}. Our result here further suggests that the stationary self-similar profiles (i.e. the homogeneous self-similar solutions) are truly stationary in the sense that they cannot be the initial data for another, time-dependent solution of the Euler equations.

We thank the Simons Foundation and New York University for their support. We thank Tim Binz and Hyungjun Choi for many collaborative discussions. We disclose the following use of LLM's: we used Gemini 3.1 Pro to proofread earlier drafts of this paper. Gemini 3.1 Pro found a gap in the previous argument and suggested a path to fill it, which inspired our Lemma \ref{CHAR}. 

\section{Homogeneous Steady States} 
Let $\alpha\in (0,2)$ be a real parameter. We consider solutions of Equation \eqref{EULER} with the following self-similar ansatz:
$$ \omega(x,t) = t^{-1}\Omega(xt^{-1/\alpha}).$$
The function $\omega(x,t)$ solves the initial value problem in Equation \eqref{EULER} with the above ansatz if and only if the function $\Omega(\xi)$ satisfies the following stationary boundary-value problem
\begin{equation}\label{SEULER}\begin{gathered}
-\Omega -\frac{\xi}{\alpha}\cdot\nabla \Omega+V\cdot\nabla \Omega=0\\
V= K_{BS}\ast \Omega\\
\Omega(\xi) \sim |\xi|^{-\alpha}\cdot f\left(\frac{\xi}{|\xi|}\right) \quad \textrm{ as } |\xi|\to \infty.
\end{gathered}
\end{equation}
The initial value problem has been replaced with a boundary value problem with a boundary condition at spatial infinity, and the limit above can be interpreted as either a point-wise limit or as the following blow-down limit in some function space:
$$ \lim_{R\to \infty} R^\alpha \cdot \Omega(R\cdot \xi) = |\xi|^{-\alpha}f\left(\frac{\xi}{|\xi|}\right).$$
These are the forward self-similar solutions of Equation \eqref{EULER}. Note that the boundary value of Equation \eqref{SEULER} (therefore the initial data of Equation \eqref{EULER}) is always a $(-\alpha)$-homogeneous function on $\mathbb{R}^2$. It follows that any stationary (in physical time) self-similar solution is necessarily a $(-\alpha)$-homogeneous function. 

Let $\overline{\omega}(r,\theta)=r^{-\alpha}f(\theta)$ be a $(-\alpha)$-homogeneous steady state of Equation \eqref{EULER} corresponding to $(1-\alpha)-$homogeneous velocity field $\overline{v}(r,\theta)$ and $(2-\alpha)-$homogeneous stream function $\overline{\psi}(r,\theta)$. As observed in \cite{LS}, if a homogeneous solution of the Euler equations corresponds to the stream function $\overline{\psi}(r,\theta) = r^{2-\alpha}g(\theta)$, then $g(\theta)$ satisfies the following ordinary differential equation with Hamiltonian structure:
$$g''(\theta) = \frac{(1-\alpha)B}{2-\alpha}g(\theta)^{\tfrac{-\alpha}{2-\alpha}} - (2-\alpha)^2g(\theta),$$
for some constant $B$. We should interpret the differential equation above on every sub-interval of $\theta$ for which $g(\theta)$ does not vanish. Due to the scaling invariance of the Euler equations, we may consider only the case when $|B|=1$ (cf. \cite{LS}). The angular behavior of the vorticity is therefore
\begin{equation}\label{VORSTREAM}f= g'' + (2-\alpha)^2 g =  \frac{(1-\alpha)B}{2-\alpha}g(\theta)^{\tfrac{-\alpha}{2-\alpha}}.\end{equation}
From this we observe that 
$$\overline{\omega} = F(\overline{\psi}),$$
where $F$ is the function
$$F(x) = \frac{(1-\alpha)B}{2-\alpha} \textrm{sgn}(x)\cdot |x|^{\tfrac{-\alpha}{2-\alpha}}.$$
Let us also document what the background velocity $\overline{v}$ is in polar coordinates:
$$\overline{v}(r,\theta) = (\nabla^\perp \overline{\psi})(r,\theta) = r^{1-\alpha}\left( (2-\alpha)g(\theta) \textbf{e}_\theta -g'(\theta) \textbf{e}_r\right).$$
With the sign conventions we have made, the following relationships hold:
$$\overline{\omega} = \textrm{curl } \overline{v} = \Delta \overline{\psi}, \quad \textrm{ and } \quad \nabla^\perp \overline{\psi} = \overline{v}.$$
We also remark that the pressure corresponding to a homogeneous solution is always given by $r^{2-2\alpha}P$ for some real constant $P$. Lastly, we observe that if $\omega(r,\theta) = r^{-\alpha}f(\theta)$ as above and $f(\theta)$ is not a constant function (i.e. not the radial power law vortex), then $f$ is singular whenever $g$ has a zero. In particular, for any such homogeneous solution, singularities of the vorticity occur along rays from the origin. See \cite{CC} for further discussion on this point.

Before ending this section, we remark that when $\alpha\in(0,3/2)\cup(3/2,2)$, any homogeneous solution that is not the radial power-law vortex has stream function vanishing along some rays from the origin: this was proved in \cite{LS}. In addition, when $\alpha\in (0,1)$, $g(\theta)$ has a simple zero at any ray $\theta=\theta_0$ where $\overline{\omega}$ is singular, again proven in \cite{LS}. Otherwise, when $\alpha\in (1,3/2)\cup(3/2,2)$ we use the following conservation law (cf. Equation (21) from \cite{LS}):
\begin{equation}\label{CONS}B = (2P+(2-\alpha)^2g^2+(g')^2)g^{(2\alpha-2)/(2-\alpha)},\end{equation}
to see that $g'\sim g^{(1-\alpha)/(2-\alpha)}\to \infty$ as $\theta\to \theta_0$, which in turn implies that $g\sim (\theta-\theta_0)^{2-\alpha}$ as $\theta\to \theta_0$. However, this means that $\overline{\omega} \sim r^{-\alpha}(\theta-\theta_0)^{-\alpha}$, which is not in $L^1_{loc}(\mathbb{R}^2)$ when $\alpha>1$. Thus, since one needs to take at least a distributional derivative of $\overline{\omega}$ and multiply by an $H^1_{loc}$ function $v$ in the definition of $L_\phi$ (i.e. it includes the term $u\cdot\nabla\overline{\omega}$), it is unclear how to correctly interpret the linearization of the Euler equations around $\overline{\omega}$ in this case.

\section{Linearization}

We denote the linearization of the Euler equations around the background vorticity profile $\overline{\omega}$ by $L_\phi$, which is given by
$$L_\phi \omega : = -\overline{v}\cdot\nabla\omega - v\cdot\nabla\overline{\omega}.$$
We observe that 
$$L_\phi = L_0 + K,$$
where 
$L_0$ is the following transport operator:
$$L_0\omega = -\overline{v}\cdot\nabla \omega$$
and
$$K\omega = -(K_{BS}\ast\omega)\cdot\nabla\overline{\omega}.$$
The natural domain to consider for $L_\phi$ is 
$$D(L_\phi) = \{\omega\in L^2_m(\mathbb{R}^2) : L_\phi\omega \in L^2_m(\mathbb{R}^2)\}.$$
By our choice that the background is $(m\cdot l)$-fold symmetric for some integer $l\geq 1$, we know, at the very least, that $L_\phi \omega$ is $m$-fold symmetric when $\omega$ is.
It was proven in \cite{BC} that $L_0$ generates a strongly continuous semigroup on $L^2_m(\mathbb{R}^2)$ as long as the homogeneity parameter $\alpha\in (0,1)$.
Also note that $L_0$ is an operator that takes derivatives of $\omega$, while $K$ is a nonlocal operator that "gains" a derivative, so it is impossible that $L_0$ is relatively bounded with respect to $K$. This is a remark we make more precise in
\begin{proposition}
\label{NOTNOT}
$L_0$ is not relatively bounded with respect to $K$.
\end{proposition}
\begin{proof}
Suppose the contrary, then there exist constants $a,b>0$ such that
\begin{equation}\label{CONTRADICT}\|L_0 \omega \|_{L^2}\leq a \|K \omega \|_{L^2} + b\|\omega\|_{L^2}\end{equation}
for all $\omega\in D(K)$. Now consider some particular function $\psi$ compactly supported away from the singularities of $\overline{u}$, $\overline{\omega}$, and $\nabla \overline{\omega}$ (which are either along some rays from the origin or simply the origin itself). Let us also choose some $m$-fold rotationally symmetric $\psi$ so that $\omega= \Delta \psi$ and $\psi \in \left(C^2_c(\mathbb{R}^2)\cap H^2(\mathbb{R}^2)\right)\setminus H^3(\mathbb{R}^2)$, so that $\psi, \omega = \Delta\psi, v=\nabla^\perp\psi = K_{BS}\ast \omega$ are all compactly supported away from the singular rays and at least continuous. Then $\omega \in D(K)$ and the right-hand-side of Equation (\ref{CONTRADICT}) is finite, but the left-hand-side (which involves an additional derivative of $\omega$) is infinite, achieving a contradiction.
\end{proof}

Before continuing, we need to define two Banach spaces. First we define the following subspace of $L^2_{loc}(\mathbb{R}^2)$:
$$X: = \{f : \|f\|_X := \sup_{R>0} \left(R^{-1}\|f\|_{L^2(B_R(0))}+\|\nabla f\|_{L^2(B_R(0))}\right)<\infty \}.$$
Let $\gamma$ be any ray from the origin of $\mathbb{R}^2$, and let $\gamma_R = B_R(0)\cap \gamma$ for any $R>0$. Next, we define the following subspace of $L^2_{loc}(\gamma)$:
$$Y: = \{f : \|f\|_Y := \sup_{R>0} \left(R^{-1/2}\|f\|_{L^2(\gamma_R)}\right)<\infty \}.$$
The spaces $X,Y$ are also convenient because they are Banach spaces rather than Fr\'{e}chet spaces.
Lemma 3.1 from \cite{BC} states that $(K_{BS}\ast): L^2_m(\mathbb{R}^2) \to X$ is a continuous linear operator between the two Banach spaces. Now we prove the following custom Trace Theorem:
\begin{lemma}
\label{TRACE}
The trace operator $T$ to the ray $\gamma$ is a continuous linear mapping $T: X\to Y$.
\end{lemma}
\begin{proof}
Let $R>0$ be arbitrary. Let $f\in X$ be arbitrary and let $f_R(x) = f(Rx)$. By the usual Trace Theorem from, for example, \cite{G}, we know that for some constant $C>0$:
$$\|f_R\|_{L^2(\gamma_1)} \leq C\left(\|f_R\|_{L^2(B_1(0))} + \|\nabla f_R\|_{L^2(B_1(0))}\right).$$
By a change of variables, we get:
$$R^{-1/2}\|f\|_{L^2(\gamma_R)} \leq C\left(R^{-1}\|f\|_{L^2(B_R(0))} + \|\nabla f\|_{L^2(B_R(0))}\right),$$
so by taking a supremum in $R>0$ we get $\|f\|_Y\leq C\|f\|_X$, proving the lemma.
\end{proof}

 The next proposition we prove shows that, in the case when $\overline{\omega}$ is not the power-law vortex, $K$ cannot be relatively bounded with respect to $L_0$, and in particular, $K$ is not densely defined. 

\begin{proposition}
\label{NOTPL} Let $m\geq 2$ and $\alpha\in (0,1)$.
Suppose $\overline{\omega}$ is a homogeneous steady state of the Euler equations that is not the radial power-law vortex. Then the operator $K$ is not densely defined on $L^2_m(\mathbb{R}^2)$. Consequently, we have that $K$ is not relatively bounded with respect to the operator $L_0$.
\end{proposition}
\begin{proof}
Recall that $\overline{\psi}(r,\theta)=r^{2-\alpha}g(\theta)$ and that $\overline{\omega}(r,\theta)= r^{-\alpha}f(\theta)$. Given $\omega\in L^2_m(\mathbb{R}^2)$, we let $v= K_{BS}\ast \omega$. Then
$$L_0\omega = -\overline{v}\cdot\nabla \omega = r^{1-\alpha}g'(\theta)\partial_r\omega -(2-\alpha)r^{-\alpha}g(\theta)\partial_\theta \omega$$
and
$$K\omega = -v\cdot\nabla \overline{\omega} =-r^{-1-\alpha}\left(f'(\theta)v_\theta-\alpha f(\theta) v_r\right),$$
where $v_\theta,v_r$ are the angular and radial components of $v$ respectively. We also recall from Equation \eqref{VORSTREAM} that
$$f(\theta)=\frac{(1-\alpha)B}{2-\alpha}g(\theta)^{\tfrac{-\alpha}{2-\alpha}}$$
where $|B|=1$. Now, since $\overline{\omega}$ is not the power-law vortex, the angular function $g(\theta)$ vanishes at some $\theta_0$. When, in addition, $\overline{\omega}$ is not the power-law shear flow (for which the pressure of the background flow is zero), we know that $g(\theta)$ has a "simple zero" at $\theta_0$ (i.e. $g'(\theta_0)\neq 0$ when $\alpha\in (0,1)$). Therefore, the singularity at $\theta_0$ of $f'\sim g(\theta)^{-2/(2-\alpha)}$ is at a different order than $f\sim g(\theta)^{-\alpha/(2-\alpha)}$, and the two singularities cannot cancel one another.
Thus, on any circle of radius $r>0$, we get that
$$\int_0^{2\pi} |K\omega(r,\theta)|^2 d\theta =\infty$$
is a divergent integral. In the power-law shear flow case we have $g(\theta) = |\cos(\theta)|^{2-\alpha}$ and the singularity of leading order is of the form $|\cos(\theta)|^{-1-\alpha}$: the rest of the argument above remains the same. The only possibility remaining is that 
$$\lim_{\theta\to\theta_0}\left(f'(\theta)v_\theta(r,\theta) - \alpha f(\theta)v_r(r,\theta)\right)=0,$$
 i.e. that the singularity along the ray is cancelled exactly by the velocity vector field $v$. Let us examine this scenario. First, we observe that
 \begin{equation}\label{LIM1}\lim_{\theta\to \theta_0} v_\theta(r,\theta) = \lim_{\theta\to \theta_0} \frac{\alpha f(\theta)}{f'(\theta)}v_r(r,\theta)\end{equation}
Now we see that
\begin{equation}\label{LIM2}\lim_{\theta\to\theta_0} \frac{\alpha f(\theta)}{f'(\theta)}v_r(r,\theta) = C_\alpha \frac{g(\theta_0)}{g'(\theta_0)}\lim_{\theta\to\theta_0} v_r(r,\theta),\end{equation}
in the case when the background is not the power-law shear flow (thus $g'(\theta_0)\neq 0$). In the equation above $C_\alpha$ is some constant depending on $\alpha$. Then the fraction $\tfrac{g(\theta_0)}{g'(\theta_0)}=0$. When $g(\theta)=|\cos(\theta)|^{2-\alpha}$ and we deal with the power-law shear flow, we can see that this part of the required limit still goes to zero. 

Now, if $\omega\in L^2_m(\mathbb{R}^2)$, then $v\in H^1_{loc}(\mathbb{R}^2)$ and the mapping corresponding to the Biot-Savart law is continuous from $L^2_m(\mathbb{R}^2)\to X$,  see \cite{BC}. In particular, in terms of the angle (and this is where the use of polar coordinates is helpful), we have $v(r,\theta)\in H^1([0,2\pi]_{per})\subset C^0([0,2\pi]_{per})$ for almost every $r>0$. Thus, we can combine Equation \eqref{LIM1} and Equation \eqref{LIM2} to get $\lim_{\theta\to\theta_0} v_\theta(r,\theta)=0$ for almost every $r>0$. Let $T$ be the trace operator for the ray corresponding to $\theta=\theta_0$, which we denote by $\gamma$. By our generalization of the Trace Theorem from Lemma \ref{TRACE}, we have that $T: X\to Y$ is a continuous linear operator. Above, we have shown that if $\omega\in L^2(\mathbb{R}^2)$ and $\omega\in D(K)$, then we necessarily have $Tv_\theta\in Y$ and $Tv_\theta=0$, where $v= K_{BS}\ast \omega$ and $v_\theta$ is the angular component of the velocity. Since the Biot-Savart law is a continuous operator from $L^2_m(\mathbb{R}^2)$ to $X$, and the Trace operator is continuous from $X$ to $Y$, we conclude that $D(K)$ necessarily lies in a closed and strict subspace of $L^2_m(\mathbb{R}^2)$, namely the kernel of $T\circ (K_{BS}\ast)$, which finally proves that $K$ cannot be densely defined on $L^2_m(\mathbb{R}^2)$. Consequently, we have that $K$ cannot be relatively bounded with respect to $L_0$. \end{proof}

Now we show that $L_\phi$ cannot be a densely defined operator on $L^2_m(\mathbb{R}^2)$. For this it suffices to show that if $\omega\in L^2_m(\mathbb{R}^2)$, then $L_0\omega$ cannot cancel out the non-square-integrable singularity of $K\omega$, so that for any $\omega\in D(L_\phi)$ we have $\omega \in \textrm{ker}(T\circ(K_{BS}\ast))$, a strict closed subspace and therefore not dense.

\begin{lemma}
\label{CHAR}
Let $m \geq 2$ and $\alpha \in (0,1)$. Suppose $\overline{\omega}$ is a homogeneous steady state of the Euler equations that is not the radial power-law vortex. If $\omega \in D(L_\phi) \subset L^2_m(\mathbb{R}^2)$, then the trace of the angular velocity $v_\theta$ (where $v = K_{BS}\ast\omega$) must vanish identically on any singular ray $\theta = \theta_0$. Consequently, $L_\phi$ is not densely defined on $L^2_m(\mathbb{R}^2)$.
\end{lemma}

\begin{proof}
Assume to the contrary that $D(L_\phi)$ is dense in $L^2_m(\mathbb{R}^2)$, so there necessarily exists $\omega \in D(L_\phi)$ such that the angular component of $v = K_{BS} \ast \omega$ has a non-vanishing trace on the ray $\theta = \theta_0$. By definition, $L_\phi \omega = L_0 \omega + K \omega = h \in L^2_m(\mathbb{R}^2)$. Expanding the operators in polar coordinates, we have:
\begin{equation} \label{PDE_Euler}
r^{1-\alpha}g'(\theta)\partial_r\omega -(2-\alpha)r^{-\alpha}g(\theta)\partial_\theta \omega - r^{-1-\alpha}\left(f'(\theta)v_\theta-\alpha f(\theta) v_r\right) = h(r, \theta).
\end{equation}
Because $v \in X \subset H^1_{loc}(\mathbb{R}^2)$, the trace of the velocity $v(r, \theta)$ on rays is a continuous function of $\theta$ with values in $L^2_{loc}(0, \infty; rdr)$. For $\theta$ near $\theta_0$, we define the following functions:
\begin{align*}
W(\theta) &:= \int_0^\infty \omega(r, \theta) r^{1-\alpha} \phi(r) dr, \\
A(\theta) &:= -\int_0^\infty \omega(r, \theta) \partial_r\left(r^{2-\alpha} \phi(r)\right) dr, \\
H(\theta) &:= \int_0^\infty h(r, \theta) r \phi(r) dr, \\
U(\theta) &:= \int_0^\infty v_r(r, \theta) r^{-\alpha} \phi(r) dr, \\
V(\theta) &:= \int_0^\infty v_\theta(r, \theta) r^{-\alpha} \phi(r) dr,
\end{align*}
where we have chosen $\phi$ to be smooth, compactly supported away from the origin, and satisfying $0\leq \phi\leq 1$. Moreover, since $\omega, h \in L^2(\mathbb{R}^2; rdrd\theta)$, we have by Cauchy-Schwarz that $W, A, H \in L^2([0, 2\pi])$. In addition, $V(\theta)$ and $U(\theta)$ are continuous at $\theta = \theta_0$ since $v(r,\theta)$ is continuous in $\theta$. Lastly, since $T v_\theta(\theta_0) \not\equiv 0$, we can choose $\phi \in C^\infty_c((0, \infty))$ appropriately such that
$$ V(\theta_0) := \int_0^\infty v_\theta(r, \theta_0) r^{-\alpha} \phi(r) dr \neq 0. $$

Now multiply Equation \eqref{PDE_Euler} by $r \phi(r)$ and integrate over $r \in (0, \infty)$. If we integrate by parts in the radial derivative, we get a one-dimensional distributional differential equation in $\theta$:
$$ g'(\theta) A(\theta) - (2-\alpha)g(\theta) W'(\theta) - f'(\theta) V(\theta) + \alpha f(\theta) U(\theta) = H(\theta). $$
Recall that $\theta=\theta_0$ is an isolated zero for $g(\theta)$. We can now solve for the distributional derivative $W'(\theta)$:
\begin{equation} \label{W_ODE}
W'(\theta) = \frac{g'(\theta)}{(2-\alpha)g(\theta)} A(\theta) - \frac{f'(\theta)}{(2-\alpha)g(\theta)} V(\theta) + \frac{\alpha f(\theta)}{(2-\alpha)g(\theta)} U(\theta) - \frac{1}{(2-\alpha)g(\theta)} H(\theta).
\end{equation}
We now analyze the order of the singularity as $\theta \to \theta_0$ for each term above. We may consider the case when the background is not the power-law shear flow (since the other case is straightforward), in which case we have $g(\theta) \sim (\theta - \theta_0)$. By Equation \eqref{VORSTREAM}, $f(\theta) \sim g(\theta)^{-\frac{\alpha}{2-\alpha}} \sim (\theta - \theta_0)^{\tfrac{-\alpha }{2-\alpha}}$. Consequently, $f'(\theta) \sim (\theta - \theta_0)^{\tfrac{-\alpha }{2-\alpha}-1}$. Therefore, asympotically near $\theta=\theta_0$, we have:
$$ (2-\alpha)W'(\theta) = (\theta-\theta_0)^{-1}A(\theta) - (\theta-\theta_0)^{\tfrac{\alpha-4}{2-\alpha}} V(\theta_0) +\alpha(\theta-\theta_0)^{\tfrac{-2}{2-\alpha}} U(\theta_0) -(\theta-\theta_0)^{-1}H(\theta).$$
We have constructed $\phi$ so that $V(\theta_0)\neq 0$, so 
$$|(\theta-\theta_0)^{\tfrac{\alpha-4}{2-\alpha}} V(\theta_0)| \gg |\alpha(\theta-\theta_0)^{\tfrac{-2}{2-\alpha}} U(\theta_0)|.$$
Therefore, we asymptotically have:
$$ (2-\alpha)(\theta-\theta_0)W'(\theta) = A(\theta) - (\theta-\theta_0)^{\tfrac{-2}{2-\alpha}} V(\theta_0)  -H(\theta).$$
We consider the distributional equation above on the sufficiently small interval $x\in(0,\epsilon)$, where $x= \theta-\theta_0$ is a simple change of variables. The question now becomes equivalent to whether it is possible to have:
$$xB'(x) = I(x) - V(\theta_0)x^{\tfrac{-2}{2-\alpha}},$$
where $B(x) = (2-\alpha)^{-1}W(x+\theta_0)$ and $I(x)\in L^2((0,\epsilon))$ is given by $I(x) = A(x+\theta_0)-H(x+\theta_0)$. If we divide the differential equation by $x$ and integrate from $x$ to $\epsilon$, we have:
$$B(\epsilon)-B(x) = \int_x^\epsilon \frac{I(s)}{s} ds  +V(\theta_0)\frac{2-\alpha}{2}\left(\epsilon^{\tfrac{-2}{2-\alpha}}-x^{\tfrac{-2}{2-\alpha}}\right).$$
Now, we know
$$\left| \int_x^\epsilon \frac{I(s)}{s} ds \right| \leq \|I\|_{L^2}\sqrt{x^{-1}-\epsilon^{-1}}.$$
Therefore, as $x\to 0$, we have found that $B(x) \sim x^{-2/(2-\alpha)}$ which, since $-2/(2-\alpha)<-1$ always, shows that $B\not\in L^2(dx)$, a contradiction.
\end{proof}

\section*{Statements and Declarations}
We declare that there is no conflict of interest.

\end{document}